\documentclass[english]{article}
\usepackage[T1]{fontenc}
\usepackage[latin9]{inputenc}
\usepackage{amsmath}
\usepackage{amsthm}
\usepackage{amssymb}
\usepackage{fancyhdr}
\usepackage{color}
\usepackage{bm}
\makeatletter
\theoremstyle{plain}
\newtheorem{thm}{\protect\theoremname}

\newtheorem{lem}{Lemma}

\theoremstyle{definition}
\newtheorem{rem}{Remark}

\newcommand\mbf[1]{\mathbf{#1}} 

\usepackage{extpfeil}

\newcommand{\tr}{\mathrm{tr}}

\makeatother

\usepackage{babel}
\providecommand{\theoremname}{Theorem}

\author{Joseph Wells, Mary Cook, Karleigh Pine, Benjamin D. Robinson}

\begin{document}

\title{Fisher-Rao distance on the covariance cone}

\date{}

\maketitle

\begin{abstract}
The Fisher-Rao geodesic distance on the statistical manifold consisting of zero-mean p-dimensional multivariate Gaussians appears without proof in several places (such as Steven Smith's ``Covariance, Subspace, and Intrinsic Cramer-Rao Bounds''). In this paper, we give a proof using basic Riemannian geometry.
\end{abstract}

\section{Introduction}

Information geometry is the application of ideas from differential
geometry to the field of statistics. Rao \cite{rao1945information}
was the first to observe that the Fisher information matrix forms
a Riemannian metric for certain models, inducing on them the structure
of a Riemannian manifold. Amari \cite{shun2012differential,amari1987differential,amari2007methods,amari2001information}
extended his observations considerably, discussing models which additionally
possess so-called flat dualistic affine connections. Perhaps the culmination
of his work is a novel proof of the Expectation-Maximization Theorem
within the geometric context. Another famous result from information
geometry is Cencov's theorem \cite[Theorem~11.1]{cencov2000statistical},
which characterizes all Riemmanian metrics and affine connections
which are invariant to sufficient statistics for the family of models
of all nonzero densities on finite outcome spaces. Other notable authors
who have studied information geometry are Murray and Rice \cite{murray1993differential};
Le, Ay, Jost, Schwachhofer \cite{ay2015information}; and Bauer, Bruveris,
and Michor \cite{bauer2016uniqueness}; and there have been a host
of others claiming novel applications, although these claims usually
are without proof, or are just re-proofs of facts from statistics
that were already well-known.

Another notable author who has studied information geometry is Smith
\cite{smith2005covariance}. Smith developed, using the structure of
an affine connection on a statistical model, a family of Cramer-Rao-type
lower bounds on expected squared geodesic distance for fairly general
models. Smith focused in particular on the model of a multivariate zero-mean
Gaussian. Identifying zero-mean Gaussians with their covariances naturally gives
this model the structure of the cone of positive-definite symmetric
(or Hermitian) matrices. For this model, Smith derives a Cramer-Rao-type
lower bound on the expected squared Fisher-Rao distance---the infimum of the lengths of all smooth paths between two points, where differential length corresponds to the Fisher information Riemannian metric.
Unlike the usual Cramer-Rao bound, this bound is independent of the
true underlying parameter, making it potentially more useful since
this parameter is in fact unknown. The question is why one should
care about the expected squared Fisher-Rao distance, or about Fisher-Rao distance at all? This debate usually boils down to claims about
``naturality'' of the distance, but usually has little substance.
In this article we do not engage in this discussion, but rather seek
to prove explicitly a closed-form expression for Fisher-Rao distance sketched by 
Smith, leaving the application of this distance to future work.
More precisely, we seek to prove the following Theorem, which corresponds
to the case of real scalars. (Smith's result concerns the complex case, and that case is similar.)
\begin{thm}\label{fisher-rao_dist}
  The Fisher-Rao distance between two covariance matrices $\mbf{R}$
  and $\mbf{S}\in\mathbb{R}^{p\times p}$ is given by
  \[
  d(\mbf{R},\mbf{S})^{2}=\frac{1}{2}\mathrm{tr}\left[\left(\log\mbf{R}^{-1/2}\mbf{S}\mbf{R}^{-1/2}\right)^{2}\right].
  \]

\end{thm}

\section{Proof of the main theorem}

Let $\mathcal{P}$ denote the set of $p \times p$ covariance matrices and $M$ be the model of $p$-dimensional real Gaussians with mean 0: 
\[
M=\left\{p(\mbf{x};\mbf{R})=\frac{1}{(2\pi)^{p/2}|\mbf{R}|^{1/2}} \exp\left(-\frac{1}{2} \mbf{x}^\top \mbf{R}^{-1}\mbf{x}\right): \, \mbf{R}\in \mathcal{P} \right\}. 
\]
$M$ and $\mathcal{P}$ are naturally identified.  We give $\mathcal{P}$ the smooth structure of $\mathbb{R}^{p^2}$, and since it embeds smoothly into $\mathbb{R}^{p^2}$, tangent vectors are precisely derivatives of smooth paths through $\mathcal{P}$.  This means that tangent vectors are precisely $p\times p$ symmetric matrices.  To see this, suppose $\mbf{D}$ is a symmetric $p\times p$ matrix and let $\gamma: (-\epsilon, \epsilon) \to \mathcal{P}$ be the path $\mbf{R}+t\mbf{D}$.   The derivative of this path at zero is $\mbf{D}$.  Conversely, suppose $\gamma$ is a smooth path through $\mathcal{P}$ originating from $\mbf{R}$.  Then the difference quotient
\[
 \frac{\gamma(t)-\gamma(0)}{t}
\]
is always symmetric; thus, so is its limit.

The proof of the main theorem will rely on the following lemmas:
\begin{lem}
  The Fisher information metric $g$ for the model $M$ at the point $\mbf{R}$ is given by
  \[
  g_{\mbf{R}}(\mbf{A},\mbf{B}) = \frac{1}{2}\mathrm{tr}\mbf{R}^{-1}\mbf{A}\mbf{R}^{-1}\mbf{B},
  \] 
    where $\mbf{A}, \mbf{B}$ are real symmetric matrices.

\end{lem}

\begin{proof}
Given a parametric model $\{p(\cdot;\theta)d\mu:\, \theta \in \Theta \subset \mathbb{R}^N\}$ on a sample space $\mathcal{X}$, for some connected open set $\Theta$, the Fisher information matrix is
\[
(g_\theta)_{ij} =\int_{\mathcal{X}} \frac{\partial}{\partial \theta^i} \log p(x; \theta) \frac{\partial}{\partial \theta^j} \log p(x; \theta)\,  p(x;\theta) d\mu(x), 
\]
when defined. Under suitable regularity conditions, which will hold here for the Gaussian model, this can be reexpressed as
\[
(g_\theta)_{ij} =-\int_{\mathcal{X}} \frac{\partial^2}{\partial \theta^i\partial\theta^j} \log p(x; \theta)\,  p(x;\theta) d\mu(x), 
\]
This extends by bilinarity to a 2-tensor
\begin{equation} \label{eq:quadratic-fisher-info}
g_\theta (v, w) = -\int_{\mathcal{X}} vw(\log p(x;\cdot))\, p(x; \theta)d\mu(x),
\end{equation}
where $v$ and $w$ are tangent vectors at $\theta$.

For the covariance model, let us consider $g_{\mbf{R}}(\mbf{D},\mbf{D})$, where $\mbf{D}$ is a tangent vector at $\mbf{R}$.  We abuse notation and use  the field of symmetric matrices $\mbf{D}$ interchangeably with the derivation that corresponds to it.

We have
\[
\mbf{D}(\log p(\mbf{x};\cdot))=\left.\frac{d}{dt} \log p(\mbf{x}; \mbf{R}+t\mbf{D})\right|_{t=0}.
\]
On the right side, we have 
\begin{align*}
 \log p(\mbf{x}; \mbf{R}+t\mbf{D}) = \log\frac{1}{(2\pi)^{p/2}} - \frac{1}{2} \tr \log \left( \mbf{R}+t\mbf{D}\right) -\frac{1}{2}\mbf{x}^\top \left(\mbf{R} +t\mbf{D} \right)^{-1} \mbf{x} .
\end{align*}
Taking the derivative for $t$ and setting $t=0$ gives
\[
\mbf{D}(\log p(\mbf{x}; \cdot)) = -\frac{1}{2}\tr \mbf{R}^{-1}\mbf{D} + \frac{1}{2}\mbf{x}^\top \mbf{R}^{-1}\mbf{D}\mbf{R}^{-1} \mbf{x}.
\]
Applying $\mbf{D}$ again yields
\begin{equation} \label{eq:d-squared}
\mbf{D^2}(\log p(x;\cdot)) = \frac{1}{2} \tr \mbf{R}^{-1}\mbf{D}\mbf{R}^{-1}\mbf{D} -  \mbf{x}^\top \mbf{R}^{-1}\mbf{D}\mbf{R}^{-1}\mbf{D}\mbf{R}^{-1}\mbf{x}.
\end{equation}
The latter term can be written
\[
\tr \left( \mbf{R}^{-1}\mbf{D}\mbf{R}^{-1}\mbf{D}\mbf{R}^{-1}\mbf{x}\mbf{x}^\top \right),
\]
so integrating this term against $p(x;\mbf{R})$ yields 
\[
\tr\left( \mbf{R}^{-1}\mbf{D}\mbf{R}^{-1}\mbf{D}\mbf{R}^{-1}\mbf{R} \right) = \tr \left( (\mbf{R}^{-1}\mbf{D})^2\right).
\]
Adding in the first term of \eqref{eq:d-squared} and applying the negative sign in \eqref{eq:quadratic-fisher-info} gives
\[
g_{\mbf{R}}(\mbf{D}, \mbf{D}) = \frac{1}{2} \tr \left((\mbf{R}^{-1}\mbf{D})^2\right).
\]
The result follows from polarization.
%

\end{proof}

Although it is assumed that the reader is at least somewhat familiar with Riemannian geometry, we recall a few definitions and set the tone notationally. Using $\mathfrak{X}(M)$ to denote the smooth vector fields on a smooth manifold $M$, an {\it affine connection} is a map
\begin{align*}
  \nabla: \mathfrak{X}(M) \times \mathfrak{X}(M) &\rightarrow \mathfrak{X}(M) \\
  (X,Y) &\mapsto \nabla_X Y
\end{align*}
that is $C^{\infty}(M)$-linear in the first coordinate, $\mathbb{R}$-linear in the second coordinate, and for all $f \in C^{\infty}(M)$ satisfies
\[
\nabla_X (fY) = f \nabla_X Y + (Xf) Y.
\]
Given a Riemannian metric $g$, the {\it Levi-Civita connection} is an affine connection $\nabla$ that also satisfies the following for all $X,Y,Z \in \mathfrak{X}(M)$:
\begin{enumerate}
\item $X g(Y,Z) = g(\nabla_X Y, Z) + g(Y, \nabla_X Z)$, and
\item $\nabla_X Y - \nabla_Y X = XY - YX$.
\end{enumerate}
As is well-known, the Levi-Civita connection for a given metric is the unique connection with these properties. (This is sometimes called the {\it Fundamental Theorem of Riemannian Geometry}.)

Given a local frame $(\partial_i)$ for the tangent bundle $TM$, we have that the connection coefficients satisfy
\[
\nabla_{\partial_i} \partial_j = \Gamma_{ij}^{k} \partial_k
\]
and these coefficients $\Gamma_{ij}^{k}$ are called {\it Christoffel symbols}. A path $\gamma(t)=(\gamma_{i}(t))$ in $M$ is called a {\it geodesic} if it satisfies the {\it geodesic equation}
\begin{align*}
\ddot{\gamma}_k(t) + \Gamma_{ij}^{k}\dot{\gamma}_i(t)\dot{\gamma}_j(t).
\end{align*}
Here it is understood that $\Gamma_{ij}^k$ is in $C^\infty(M)$ and that the above equation holds at $\bm{\gamma}(t)$.   It is sometimes convenient to express the quadratic terms above in vector notation:
\[
\mathbf{\Gamma}_{\bm{\gamma}(t)}\left(\dot{\bm{\gamma}}(t), \dot{\bm{\gamma}}(t)\right).
\]

Given $p\in M$ and an affine connection $\nabla$, we say that an open set $U$ containing $p$ is a \emph{normal neighborhood} of $p$ iff for all $q\in U$, the solution to the geodesic equation with boundary conditions $p$ and $q$ is unique.  It is well known 
that if $\nabla$ is the Levi-Civita connection corresponding to a Riemannian metric $g$, the geodesics within a normal neighborhood are length-minimizing.  It is also well-known that the covariance cone with the Fisher information metric has non-positive sectional curvature, making the whole cone a normal neighborhood of every point.  Thus, Fisher-Rao distance, which we have defined as an infimum of path distances, is achieved by the geodesics corresponding to the Levi-Civita connection, and we need only find the lengths of these geodesics.

We note that it is well-known that if $\nabla$ is an affine connection, $\bm{\Gamma}$ is the corresponding Christoffel symbol, and $X,Y\in \mathfrak{X}(M)$, then
\begin{equation}\label{eq:christoffel-connection}
\nabla_X Y = X Y + \bm{\Gamma}(X,Y).
\end{equation}

We are almost ready to prove our theorem, but first we need a couple of lemmas.

\begin{lem} \label{lem:deriv-of-Rinv}
If $\mbf{X}\in\mathfrak{X}(M)$ and $f(\mbf{R})=\mbf{R}^{-1}$, then $(\mbf{X}f)(\mbf{R})=-\mbf{R}^{-1}\mbf{X}\mbf{R}^{-1}$, where $\mbf{X}$ on the right side is considered as a field of symmetric matrices.
\end{lem}
\begin{proof}
Let $a$ be an invertible square matrix, $b$ be a square matrix of the same size, and let us compute the difference quotient for the derivative of $(a+hb)^{-1}$ at $h=0$:
\begin{align*}
\frac{1}{h}\left( (a+hb)^{-1}-a^{-1} \right) & = \frac{1}{h}(a+hb)^{-1}\left(1-(a+hb)a^{-1}\right) \\
& = \frac{1}{h} -(a+hb)^{-1}hba^{-1}
\end{align*}
Taking the limit as $h\to 0$, we get $-a^{-1}ba^{-1}$.

The proof is completed by fixing a base point $p\in M$ and taking $a=\mbf{R}_p$ and $b = \mbf{X}_p$.
\end{proof}

\begin{lem}
  The map $\bm{\gamma}:(-\varepsilon,\varepsilon) \rightarrow \mathcal{P}$ given by
  \begin{align}\label{covariance_geodesic}
    \bm{\gamma}(t) = \mbf{R}^{1/2}\exp\!\left(t\mbf{R}^{-1/2}\mbf{D}\mbf{R}^{-1/2}\right)\mbf{R}^{1/2}
  \end{align}
  is a geodesic emanating from $\mbf{R}$ in the direction of $\mbf{D}$. (Here $\exp$ is the usual matrix exponential.)
\end{lem}

\begin{proof}
  In this proof and what follows we will denote $\bm{\gamma}$ without boldface, and it will be understood that it is matrix-valued. 
  
  We first claim that the Christoffel symbols of the Levi-Civita connection are given by
  \begin{align*}
    \bm{\Gamma}_{\mbf{R}}(\mbf{A},\mbf{B}) &= -\frac{1}{2}\left(\mbf{A}\mbf{R}^{-1}\mbf{B} + \mbf{B}\mbf{R}^{-1}\mbf{A}\right).
  \end{align*}
  (This is stated without proof in \cite{smith2005covariance}.)
 To do this, we must prove that $\mbf{\Gamma}$ satisfies \eqref{eq:christoffel-connection} for some affine connection and verify properties 1 and 2 above.  Indeed, by defining $\nabla$ as in \eqref{eq:christoffel-connection} with our $\mbf{\Gamma}$ as above, it is straightforward to show that $\nabla$ satisfies the properties of an affine connection. Property 2 follows from the definition of $\nabla$ and the fact that $\bm{\Gamma}_{\mbf{R}}(\mbf{A},\mbf{B}) - \bm{\Gamma}_{\mbf{R}}(\mbf{B},\mbf{A}) = 0$. For property 1, let $\mbf{X},\mbf{Y},\mbf{Z} \in \mathfrak{X}(\mathcal{P})$. We then have
 \begin{align} 
	\mbf{X}g_{\mbf{R}}(\mbf{Y},\mbf{Z})
    &= -\frac{1}{2} \tr\left( \mbf{R}^{-1}\mbf{X}\mbf{R}^{-1}\mbf{Y}\mbf{R}^{-1}\mbf{Z} \right)
 +\frac{1}{2}\tr\left( \mbf{R}^{-1}\mbf{XYR}^{-1}\mbf{Z} \right) + \nonumber \\
 &\quad - \frac{1}{2} \tr\left( \mbf{R}^{-1}\mbf{Y}\mbf{R}^{-1}\mbf{X}\mbf{R}^{-1}\mbf{Z} \right) 
  +\frac{1}{2}\tr\left( \mbf{R}^{-1}\mbf{YR}^{-1}\mbf{XZ} \right) \label{eq:XonZ}
 \end{align}
 (Here we have used Lemma~\ref{lem:deriv-of-Rinv} and are again using derivations and vector fields interchangeably.)  On the other hand 
 \begin{align*}
 & g_{\mbf{R}}\left( \nabla_{\mbf{X}}\mbf{Y}, \mbf{Z}\right) \\ 
 & =
 g_{\mbf{R}}(\mbf{XY},\mbf{Z}) +  g_{\mbf{R}}(\bm{\Gamma}_{\mbf{R}}(\mbf{X},\mbf{Y}), \mbf{Z}) \\
 &= \frac{1}{2}\tr\left(\mbf{R}^{-1}\mbf{XY}\mbf{R}^{-1}\mbf{Z}\right) + \frac{1}{2} \tr\left( \mbf{R}^{-1}\left(-\frac{1}{2}\right)\left( \mbf{X}\mbf{R}^{-1}\mbf{Y} +\mbf{Y}\mbf{R}^{-1}\mbf{X}\right)\mbf{R}^{-1}\mbf{Z} \right)
 \end{align*}
 and
  \begin{align*}
& g_{\mbf{R}}\left( \mbf{Y},\nabla_{\mbf{X}} \mbf{Z}\right) \\ 
 & =
 g_{\mbf{R}}(\mbf{Y},\mbf{XZ}) +  g_{\mbf{R}}(\mbf{Y}, \mbf{\Gamma}_{\mbf{R}}(\mbf{X},\mbf{Z})) \\
 &= \frac{1}{2}\tr\left(\mbf{R}^{-1}\mbf{Y}\mbf{R}^{-1}\mbf{XZ}\right) + \frac{1}{2} \tr\left( \mbf{R}^{-1} \mbf{Y} \mbf{R}^{-1} \left(-\frac{1}{2}\right)\left( \mbf{X}\mbf{R}^{-1}\mbf{Z} +\mbf{Z}\mbf{R}^{-1}\mbf{X}\right) \right)
 \end{align*}
  Using the circulant property of trace, the sum of these two expressions is easily seen to equal  
\eqref{eq:XonZ}.
  
  We now wish to prove that $\gamma=\gamma(t)$ satisfies the geodesic equation
  \begin{align}\label{geodesic_eq}
    \ddot{\gamma} + \bm{\Gamma}_{\gamma}(\dot{\gamma},\dot{\gamma}) &= 0.
  \end{align}

  For simplicity, write $\mbf{E} = \exp\!\left(t\mbf{R}^{-1/2}\mbf{D}\mbf{R}^{-1/2}\right)$. Then we have that
  \[
  \frac{d\mbf{E}}{dt} = \mbf{R}^{-1/2}\mbf{D}\mbf{R}^{-1/2}\mbf{E} = \mbf{E}\mbf{R}^{-1/2}\mbf{D}\mbf{R}^{-1/2}
  \]
  whence
  \begin{align*}
    \gamma &= \mbf{R}^{1/2}\mbf{E}\mbf{R}^{1/2}, \\
    \dot{\gamma} &= \mbf{D}\mbf{R}^{-1/2}\mbf{E}\mbf{R}^{1/2} = \mbf{R}^{1/2}\mbf{E}\mbf{R}^{-1/2}\mbf{D}, \\
    \ddot{\gamma} &= \mbf{D}\mbf{R}^{-1/2}\mbf{E}\mbf{R}^{-1/2}\mbf{D}.
  \end{align*}
  It is then straightforward to compute
  \[
  \ddot{\gamma} + \bm{\Gamma}_{\gamma}(\dot{\gamma},\dot{\gamma}) = \ddot{\gamma} - \dot{\gamma} \gamma^{-1} \dot{\gamma} = 0.
  \]
\end{proof}

\begin{proof}[Proof of Theorem \ref{fisher-rao_dist}]
  Let $\gamma = \gamma(t)$ be a geodesic from $\mbf{R}$ to $\mbf{S}$. We may parameterize $\gamma$ so that $\gamma(0) = \mbf{R}$ and $\gamma(1) = \mbf{S}$. Since $\gamma$ takes the form of the map in Equation (\ref{covariance_geodesic}), we solve for $\mbf{D}$ in the equation $\mbf{S}=\gamma(1)$ to get
  \[
  \mbf{D} = \mbf{R}^{1/2} \mbf{L} \mbf{R}^{1/2}
  \]
  where
  \[
  \mbf{L} = \log\mbf{R}^{-1/2}\mbf{S}\mbf{R}^{-1/2}.
  \]
  We then have that
  \begin{align*}
    g_{\gamma}(\dot{\gamma},\dot{\gamma}) 
    &= \frac{1}{2}\tr\!\left[\left(\gamma^{-1}\dot{\gamma}\right)^2\right] \\
    &= \frac{1}{2}\tr\!\left[\left(\mbf{R}^{-1}\mbf{D}\right)^2\right] \\
    &= \frac{1}{2}\tr\!\left[\left(\mbf{R}^{-1/2}\mbf{L}\mbf{R}^{1/2}\right)^2\right] \\
    &= \frac{1}{2}\tr\!\left[\mbf{R}^{-1/2}\mbf{L}^{2}\mbf{R}^{1/2}\right] \\
    &= \frac{1}{2}\tr\!\left[\mbf{L}^{2}\right] \\
    &= \frac{1}{2}\tr\!\left[\left(\log\mbf{R}^{-1/2}\mbf{S}\mbf{R}^{-1/2}\right)^{2}\right] \\
  \end{align*}
  The distance between $\mbf{R}$ and $\mbf{S}$ is just the length of the geodesic segment between these two covariance matrices 
  \[
  d(\mbf{R},\mbf{S}) = \int_{0}^{1} \sqrt{g_{\gamma}(\dot{\gamma},\dot{\gamma})}\,dt = \sqrt{\frac{1}{2}\tr\!\left[\left(\log \mbf{R}^{-1/2}\mbf{S}\mbf{R}^{-1/2}\right)^{2}\right]}.
  \]
\end{proof}

\begin{rem}
Note that this distance is consistent (up to a factor of $1/\sqrt{2}$) with the distance for complex scalars derived in \cite{smith2005covariance} for complex scalars.  Smith's distance squared is
\[
\sum_j (\log \lambda_j )^2,
\]
where $\lambda_j$ are the generalized eigenvalues of the pencil $\mbf{S}-\lambda\mbf{R}$.  These may be defined as the roots of the following polynomial in $\lambda$: $\det(\mbf{S}-\lambda\mbf{R})$. But these roots are the same as those of $\det(\mbf{R}^{-1/2}\mbf{S}\mbf{R}^{-1/2} - \lambda 1)$; thus, the sum of the squares of the logs of these eigenvalues is precisely the trace appearing in the last theorem.
\end{rem}

\newpage

\bibliographystyle{plain}
\bibliography{information-geometry-bib}

\end{document}